\documentclass[a4paper,12pt]{article}

\usepackage{amsmath, amsthm, amsfonts, amssymb}

\theoremstyle{plain}
\newtheorem{theorem}{Theorem}[section]

\newtheorem{lemma}[theorem]{Lemma} 

\theoremstyle{remark}

\theoremstyle{definition}

\numberwithin{equation}{section}

\def\Om{\Omega}

\def\e{\varepsilon}
\def\g{\gamma}

\def\l{\lambda}
\def\p{\partial}

\def\a{\alpha}
\def\b{\beta}

\def\d{\delta}
\def\L{\Lambda}
\def\z{\zeta}

\def\Odr{\mathcal{O}}
\def\H{W_2}

\def\Hinf{W_\infty}

\def\di{\,\mathrm{d}}

\DeclareMathOperator{\RE}{Re}

\DeclareMathOperator{\discspec}{\sigma_{disc}}

\DeclareMathOperator{\dist}{dist}
\DeclareMathOperator{\supp}{supp}

\DeclareMathOperator{\rank}{rank}

\DeclareMathOperator*{\esssup}{ess\,sup}

\begin{document}

\allowdisplaybreaks


\begin{center}
\textbf{\Large Asymptotics of the eigenvalues of elliptic
systems with fast oscillating coefficients}

\bigskip
\bigskip

{\large D. Borisov}\footnotetext[1]{ The work is supported in
parts by RFBR (07-01-00037). The author is also supported by
\emph{Marie Curie International Fellowship} within 6th European
Community Framework Programm (MIF1-CT-2005-006254). The author
gratefully acknowledges the support from Deligne 2004 Balzan
prize in mathematics. }

\end{center}

\begin{abstract}
We consider singularly perturbed second order elliptic system in
the whole space with fast oscillating coefficients. We construct
the complete asymptotic expansions for the eigenvalues
converging to the isolated ones of the homogenized system, as
well as the complete asymptotic expansions for the associated
eigenfunctions.
\end{abstract}

\section*{Introduction}

Many works are devoted to the studying of the asymptotic
behaviour of the elliptic operators in bounded domains (see, for
instance, \cite{BP, ZhKO} and the references therein). The
similar question for unbounded domains are studied much less.
Recently quite intensive study of the such problems has been
initiated (see  \cite{AA, BS, BS5, Zh3} and the references
therein). One of the interesting question concerns the behaviour
of the spectrum of the mentioned operators in unbounded domains
treated as the operators in  $L_2$. The one-dimensional case was
studied in \cite{BG, B3, B1, BG2}. There we considered the
operators whose coefficients depended on slow and fast variable.
The dependence of the slow one was supposed to localized on the
finite interval, i.e., at infinity the coefficients depended
either on fast variable or were constant. We have studied in
detail the asymptotic behavior of continuous and point spectrum
and constructed the asymptotic expansions for the eigenvalues
and the associated eigenfunctions, as well as for the edges of
the zones of the continuous spectrum.

In the present work we generalize partially the results of the
cited papers to the multi-dimensional case.  Namely, we consider
a second order elliptic system in a multi-dimensional space with
fast oscillating coefficients. The coefficients depend on slow
and fast variable and are periodic w.r.t. the fast one, while
they are uniformly bounded together with all their derivatives
w.r.t. the slow variable. The main result of the paper is the
asymptotic expansions for the eigenvalues of the perturbed
system converging to the isolated eigenvalues of the homogenized
one. Moreover, we construct the complete asymptotic expansions
for the corresponding eigenfunctions.  We also note that the
similar elliptic system has already been studied in \cite{AA},
and the leading terms of the asymptotic expansion for the
resolvent were constructed. It has been also shown in \cite{AA}
that basic operators of the mathematical physics are particular
cases of this elliptic system; the great number of interesting
examples has been adduced in \cite{BS, BS5}. The results of our
work are applicable for all these examples; in the case of the
examples in \cite{BS, BS5} we also allow the coefficients to
depend on the slow variable.

\section{Formulation of the problem and main results}

Let $Y$ be a Banach space. By $\H^k(\mathbb{R}^d;Y)$ and
$\H^k(\mathbb{R}^d;Y)$, $d\geqslant 1$, we denote the Sobolev
space of the functions defined on $\mathbb{R}^d$ with values in
$Y$ possessing the finite norms
\begin{equation*}
\|\mathbf{u}\|_{\Hinf^k(\mathbb{R}^d;Y)}:=\max
\limits_{\genfrac{}{}{0 pt}{}{\a\in
\mathbb{Z}_+^d}{|\a|\leqslant k}}\,\esssup\limits_{x\in
\mathbb{R}^d} \Big\|\frac{\p^{|\a|} \mathbf{u}}{\p
x^a}\Big\|_{Y},\quad
\|\mathbf{u}\|_{\H^k(\mathbb{R}^d;Y)}^2=\sum\limits_{\genfrac{}{}{0
pt}{}{\a\in \mathbb{Z}_+^d}{|\a|\leqslant k}}
\int\limits_{\mathbb{R}^d} \Big\|\frac{\p^{|\a|} \mathbf{u}}{\p
x^a}\Big\|_{Y}^2 dx.
\end{equation*}
If $k=0$, we will employ the notations
$L_\infty(\mathbb{R}^d;Y):=\Hinf^0(\mathbb{R}^d;Y)$,
$L_2(\mathbb{R}^d;Y):=\H^0(\mathbb{R}^d;Y)$. We denote
$\mathcal{W}(\mathbb{R}^d;Y):=\bigcap\limits_{i=1}^\infty
\Hinf^k(\mathbb{R}^d;Y)$.

In the space $\mathbb{R}^d$ we select a periodic lattice with
the elementary cell  $\square$. We will employ the symbol
$C_{per}^\g(\overline{\square})$ to indicate the space of
$\square$-periodic functions having finite H\"older norm
$\|\cdot\|_{C_{per}^\g(\overline{\square})}:=
\|\cdot\|_{C^\g(\overline{\square})}$.  In $\mathbb{R}^d$ we
introduce the Cartesian coordinates $x=(x_1,\ldots,x_d)$ and
$\xi=(\xi_1,\ldots,\xi_d)$. We will often treat a
$\square$-periodic w.r.t. $\xi$ vector-function
$\mathbf{f}=\mathbf{f}(x,\xi)$ as mapping points
$x\in\mathbb{R}^d$ into the function $f=f(x,\cdot)$. It will
allow us to speak about the belonging of $\mathbf{f}(x,\xi)$ to
$\Hinf^k(\mathbb{R}^d;C^\g_{per}(\overline{\square}))$ and
$\H^k(\mathbb{R}^d;C^\g_{per}(\overline{\square}))$.

Let $A=A(x,\xi)\in\mathcal{W}(\mathbb{R}^d;C_{per}^{1+\b}
(\overline{\square}))$ be a matrix-valued function of the size
$m\times m$, $m\geqslant 1$, $\b\in(0,1)$. We assume that it is
hermitian and satisfies the uniform in  $(x,\xi)\in
\mathbb{R}^{2d}$ estimate
\begin{equation*}
c_1 E_m\leqslant A(x,\xi)\leqslant c_2 E_m,
\end{equation*}
where $E_m$ is $m\times m$ unit matrix. By $B=B(\z)$ we denote
the matrix-valued function $B(\z)=\sum_{i=1}^{d}B_i\z_i$, where
$\z=(\z_1,\ldots,\z_d)$, and $B_i$ are constant $m\times n$
matrices, and $m\geqslant n$, $\rank B(\z)=n$, $\z\not=0$. Let
$V=V(x,\xi)\in
\mathcal{W}(\mathbb{R}^d;C^\b_{per}(\overline{\square}))$,
$a_i=a_i(x,\xi)\in
\mathcal{W}(\mathbb{R}^d;C^{1+\b}_{per}(\overline{\square}))$,
$b_i=b_i(x)\in \mathcal{W}(\mathbb{R}^d)$ are matrix-valued
functions of the sizes $n\times n$, and the matrix $V$ is
assumed to be hermitian. The entries of all the matrices are
supposed to be complex-valued. By $\e$ we denote a small
positive parameter, and given function $f(x,\xi)$ we let
$f_\e(x):=f\left(x,x/\e\right)$.

The perturbed operator is introduced as
\begin{equation*}
\mathcal{H}_\e:=B(\p)^*A_\e B(\p)+a_\e(x,\p)+V_\e 
\end{equation*}
in $L_2(\mathbb{R}^d;\mathbb{C}^n)$ with the domain
$\H^2(\mathbb{R}^d;\mathbb{C}^n)$. Here
\begin{gather*}
B(\p):=\sum\limits_{i=1}^{d} B_i\p_i,\quad
B(\p)^*:=-\sum\limits_{i=1}^{d} B_i^*\p_i,
\\
a_\e(x,\p):=a\left(x,\frac{x}{\e},\p\right),\quad a(x,\xi,\z):=
\sum\limits_{i=1}^{d} \big(a_i(x,\xi)\z_{i}b_i(x)-
b_i^*(x)\z_{i}a_i^*(x,\xi)\big), \nonumber
\end{gather*}
where  $\p=(\p_1,\ldots,\p_d)$, $\p_i$ is the derivative w.r.t.
$x_i$, the superscript $^*$ indicates hermitian conjugation. It
was shown in \cite{AA} that the operator $\mathcal{H}_\e$ is
self-adjoint and lower-semibounded uniformly in $\e$, and the
homogenized operator was obtained. Let us describe the latter.

Let $\L_0=\L_0(x,\xi),\L_1=\L_1(x,\xi)$ be the matrices of the
size $n\times n$ and $n\times m$, respectively, being
$\square$-periodic w.r.t. $\xi$ solutions of the equations
\begin{equation}\label{1.10}
\begin{aligned}
&B(\p_\xi)^*A(x,\xi)B(\p_\xi)\L_0(x,\xi)-\sum\limits_{i=1}^{d}
b_i^*(x)\frac{\p a_i^*}{\p\xi_i}(x,\xi)=0,&&
(x,\xi)\in\mathbb{R}^{2d},
\\
&B(\p_\xi)^*A(x,\xi)\big(B(\p_\xi)\L_1(x,\xi)+E_m\big)=0,&&
(x,\xi)\in\mathbb{R}^{2d},
\end{aligned}
\end{equation}
and satisfying the conditions
\begin{equation}\label{1.11}
\int\limits_\square \L_i(x,\xi)\di\xi=0,\quad x\in\mathbb{R}^d,
\qquad i=0,1.
\end{equation}
Here
$\p_\xi=\left(\frac{\p}{\p\xi_1},\ldots,\frac{\p}{\p\xi_d}\right)$.
It was established in \cite{AA} that the problems (\ref{1.10}),
(\ref{1.11}) are uniquely solvable and $\L_i\in
\Hinf^1(\mathbb{R}^d;C^{2+\b}_{per}(\overline{\square}))$. The
homogenized operator $\mathcal{H}_0$ was determined as follows
\begin{align*}
&\mathcal{H}_0:=B(\p)^*A_2 B(\p)+A_1(x,\p)+A_0, 
\\
&A_2(x):=\frac{1}{|\square|}\int\limits_\square
A(x,\xi)\big(B(\p_\xi)\L_1(x,\xi)+E_m\big)\di\xi,
\\
&A_1(x,\p):=\frac{1}{|\square|}B(\p)^*\int\limits_\square
A(x,\xi) B(\p_\xi)\L_0(x,\xi)\di\xi
\\
&\hphantom{A_1(x,\p):=}+\left(\frac{1}{|\square|}\int\limits_\square
\big(B(\p_\xi)\L_0(x,\xi)\big)^* A(x,\xi)\di\xi\right)
B(\p)+\frac{1}{|\square|}\int\limits_\square a(x,\xi,\p)\di\xi,
\\
&A_0(x):=-\frac{1}{|\square|}\int\limits_\square
\big(B(\p_\xi)\L_0(x,\xi)\big)^*A(x,\xi)B(\p_\xi)\L_0(x,\xi)\di\xi
+\frac{1}{|\square|}\int\limits_\square V(x,\xi)\di\xi,
\end{align*}
and considered as an operator in
$L_2(\mathbb{R}^d;\mathbb{C}^n)$ with the domain
$\H^2(\mathbb{R}^d;\mathbb{C}^n)$. It was shown that this
operator is self-adjoint and lower-semibounded.

Let $\l_0$ be a $N$-multiple isolated eigenvalue of
$\mathcal{H}_0$. It follows from \cite[Corollary 1.2]{AA} that
there exist exactly $N$ eigenvalues $\l_\e^{(i)}$,
$i=1,\ldots,N$, of  $\mathcal{H}_\e$ (counting multiplicity)
converging to $\l_0$ as $\e\to+0$. The aim of the paper is to
construct the complete asymptotic expansions for these
eigenvalues and the associated eigenfunctions. Before presenting
our main result, we introduce additional notations.

Let  $\boldsymbol{\psi}_0^{(i)}$ be the orthonormalized in
$L_2(\mathbb{R}^d;\mathbb{C}^n)$ eigenfunctions associated with
$\l_0$. We introduce the matrix $T$ with the entries
\begin{align*}
&T_{ij}:=\frac{1}{|\square|} \big(\mathcal{K}_{-1}(\L_1
B(\p_x)+\L_0)\boldsymbol{\psi}_0^{(i)},(\L_1
B(\p_x)+\L_0)\boldsymbol{\psi}_0^{(j)}
\big)_{L_2(\mathbb{R}^d\times\square;\mathbb{C}^n)}
\\
&\hphantom{T_{ij}:=}+\frac{1}{|\square|}
\big(\boldsymbol{\psi}_0^{(i)},\mathcal{K}_{0} (\L_1
B(\p_x)+\L_0) \boldsymbol{\psi}_0^{(j)}
\big)_{L_2(\mathbb{R}^d\times\square;\mathbb{C}^n)}
\\
&\hphantom{T_{ij}:=} +\frac{1}{|\square|}
\big(\mathcal{K}_{0}(\L_1
B(\p_x)+\L_0)\boldsymbol{\psi}_0^{(i)},
\boldsymbol{\psi}_0^{(j)}
\big)_{L_2(\mathbb{R}^d\times\square;\mathbb{C}^n)},
\\
&\mathcal{K}_{-1}:=B(\p_\xi)^*A B(\p_x)+B(\p_x)^*A B(\p_\xi)+
a(x,\xi,\p_\xi),
\\
&\mathcal{K}_0:=B(\p_x)^*A B(\p_x)+a(x,\xi,\p_x)+V,
\end{align*}
where $\p_x:=\left(\frac{\p}{\p x_1},\ldots,\frac{\p}{\p
x_d}\right)$, $\p_\xi:=\left(\frac{\p}{\p
\xi_1},\ldots,\frac{\p}{\p\xi_d}\right)$, $\frac{\p}{\p x_i}$,
$\frac{\p}{\p\xi_i}$ are respectively the partial derivatives
w.r.t. $x_i$ and $\xi_i$ for the functions $u=u(x,\xi)$. In
formulas given the arguments of all the functions except
$\boldsymbol{\psi}_0^{(j)}$ are $(x,\xi)$.

The matrix $T$ being hermitian, there exists a unitary matrix
$S_0$ such that the matrix $S_0 T S_0^*$ is diagonal. We denote
$\boldsymbol{\Psi}_0^{(i)}:=
\sum_{j=1}^{N}S_{ij}^{(0)}\boldsymbol{\psi}_0^{(j)}$, where
$S_{ij}^{(0)}$ are the elements of $S_0$. The vector-functions
$\boldsymbol{\Psi}_0^{(i)}$ are orthonormalized in
$L_2(\mathbb{R}^d;\mathbb{C}^n)$.  By $\tau_i$, $i=1,\ldots,N$,
we denote the eigenvalues of $T$.

\begin{theorem}\label{th1.3}
Let the eigenvalues of $T$ be different. Then the eigenvalues
$\l_\e^{(i)}$ satisfy the asymptotic expansions
\begin{align}
&\l_\e^{(i)}=\l_0+\sum\limits_{j=1}^{\infty}\e^j\l_j^{(i)},
\label{1.16}
\\
&\l_1^{(i)}=\tau_i,\label{1.17}
\end{align}
where the rest of the coefficients is determined by
Lemma~\ref{lm5.2}. The eigenfunctions associated with
$\l_\e^{(i)}$ can be chosen so that in the norm of
$\H^2(\mathbb{R}^d;\mathbb{C}^n)$ they satisfy the asymptotic
expansions
\begin{align}
&\boldsymbol{\psi}_\e^{(i)}(x)=
\boldsymbol{\Psi}_0^{(i)}(x)+\sum\limits_{j=1}^{N}\e^j
\boldsymbol{\Psi}_j^{(i)}\left(x,\frac{x}{\e}\right),\label{1.18}
\\
&\boldsymbol{\Psi}_1^{(i)}(x,\xi)=\big(\L_1(x,\xi)B(\p_x)+
\L_0(x,\xi)\big)\boldsymbol{\Psi}_0^{i}(x)+
\boldsymbol{\phi}_1^{(i)}(x),\label{1.19}
\end{align}
where $\boldsymbol{\phi}_1^{(i)}$ are given by (\ref{5.9}),
(\ref{5.8}), (\ref{5.12}). The rest of the coefficients in
(\ref{1.18}) is determined in Lemma~\ref{lm5.2}.
\end{theorem}

We stress that the assumption $\tau_i\not=\tau_j$, $i\not=j$, is
not essential for the constructing of the asymptotic expansions
for the eigenvalues and eigenfunctions of $\mathcal{H}_\e$. We
have used it just to simplify certain technical details. If this
assumption does not hold, the technique employed in the proof of
Theorem~\ref{th1.3} allows us to construct the asymptotics for
$\l_\e^{(i)}$ and $\boldsymbol{\psi}_\e^{(i)}$. We also note
that this assumption is the general case, if $\l_0$ is multiple,
and is surely to hold true, if $\l_0$ is simple.

\section{Auxiliary statements}

In the present section we prove a series of auxiliary
statements.

\begin{lemma}\label{lm2.2}
For any $\mathbf{u}\in\H^2(\mathbb{R}^d;\mathbb{C}^n)$ the
uniform in $\e$ estimate
\begin{equation*}
\|\mathbf{u}\|_{\H^2(\mathbb{R}^d;\mathbb{C}^n)}\leqslant
C\e^{-5}
\left(\|\mathcal{H}_\e\mathbf{u}\|_{L_2(\mathbb{R}^d;\mathbb{C}^n)}
+ \|\mathbf{u}\|_{L_2(\mathbb{R}^d;\mathbb{C}^n)}\right)
\end{equation*}
holds true.
\end{lemma}

\begin{proof}
It is sufficient to show the estimate for the vector-functions
$\mathbf{u}\in C_0^\infty(\mathbb{R}^d)$ since the latter set is
dense in $\H^2(\mathbb{R}^d;\mathbb{C}^n)$. Throughout the proof
we indicate by $C$ inessential constants independent of $\e$.

We denote $\mathbf{f}:=\mathcal{H}_\e\mathbf{u}$. Since
\begin{align*}
(\mathbf{f},\mathbf{u})_{L_2(\mathbb{R}^d;\mathbb{C}^n)}&
=\big(A_\e B(\p)\mathbf{u},
B(\p)\mathbf{u}\big)_{L_2(\mathbb{R}^d;\mathbb{C}^m)}
\\
&+ 2\RE\sum\limits_{i=1}^{d}\left(a_{i,\e} \p_{i}b_i \mathbf{u},
\mathbf{u}\right)_{L_2(\mathbb{R}^d;\mathbb{C}^n)} +
\big(V_\e\mathbf{u},\mathbf{u}\big)_{L_2(\mathbb{R}^d;\mathbb{C}^n)},
\end{align*}
it follows from an uniform in $\e$ inequality
\begin{equation*}
C_1\|\nabla \mathbf{u} \|_{L_2(\mathbb{R}^d;\mathbb{C}^n)}^2
\leqslant \big(A_\e
B(\p)\mathbf{u},B(\p)\mathbf{u}\big)_{L_2(\mathbb{R}^d;\mathbb{C}^m)}
\leqslant C_2 \|\nabla\mathbf{u}
\|_{L_2(\mathbb{R}^d;\mathbb{C}^n)}^2,
\end{equation*}
proven in Lemma~2.1 in \cite{AA} that
\begin{equation}\label{2.2}
\|\mathbf{u}\|_{\H^1(\mathbb{R}^d;\mathbb{C}^n)}\leqslant C
\left( \|\mathbf{f}\|_{L_2(\mathbb{R}^d;\mathbb{C}^n)}+
\|\mathbf{u}\|_{L_2(\mathbb{R}^d;\mathbb{C}^n)}\right).
\end{equation}
By the definition of $A_\e$ and the estimate established we
obtain
\begin{equation}\label{2.3b}
\begin{aligned}
\bigg\|\sum\limits_{i,j=1}^{d}B_i^*A_\e B_j \p_{ij}
\mathbf{u}\bigg\|_{L_2(\mathbb{R}^d;\mathbb{C}^n)}& \leqslant
\|\mathbf{f}\|_{L_2(\mathbb{R}^d;\mathbb{C}^n)}+
C\e^{-1}\|\mathbf{u}\|_{\H^1(\mathbb{R}^d;\mathbb{C}^n)}
\\
&\leqslant
C\e^{-1}\left(\|\mathbf{f}\|_{L_2(\mathbb{R}^d;\mathbb{C}^n)}+
\|\mathbf{u}\|_{L_2(\mathbb{R}^d;\mathbb{C}^n)}\right).
\end{aligned}
\end{equation}

Let $\sum\limits_{p} \chi_p^2(x)=1$ be a partition of unity for
$\mathbb{R}^d$ such that each of cut-off functions obeys an
inequality $0\leqslant \|\chi_p\|_{C^2(\supp\chi_p)}\leqslant
C$, where the constant $C$ is independent of $p$, and the
support of each $\chi_p$ can be shifted inside a fixed bounded
domain independent of $p$. We also assume that the number of the
functions $\chi_p$ not vanishing at a point $x\in\mathbb{R}^d$
is bounded uniformly in $x\in\mathbb{R}^d$. We denote
\begin{equation*}
\mathbf{u}_p(x):=\chi_p\left(\frac{x}{\e^2}\right)\mathbf{u}(x),
\quad \mathbf{f}_p:=-\sum\limits_{i,j=1}^{d}B_i^*  A_\e B_j
\p_{ij}\mathbf{u}_p.
\end{equation*}
We observe that by  (\ref{2.2}), (\ref{2.3b})
\begin{equation}\label{2.3a}
\begin{aligned}
\sum\limits_{p} \|\mathbf{f}_p\|_{L_2(\Om_{p,\e})}^2&\leqslant
C\e^{-9} \sum\limits_{p}
\left(\|\chi_p\mathbf{f}\|_{L_2(\mathbb{R}^d;\mathbb{C}^n)}^2+
\|\mathbf{u}_p\|_{L_2(\mathbb{R}^d;\mathbb{C}^n)}^2\right)
\\
&= C\e^{-9}
\left(\|\mathbf{f}\|_{L_2(\mathbb{R}^d;\mathbb{C}^n)}^2+
\|\mathbf{u}\|_{L_2(\mathbb{R}^d;\mathbb{C}^n)}^2\right),
\end{aligned}
\end{equation}
where $\Om_{p,\e}:=\supp\chi_p\left(\frac{x}{\e^2}\right)$. The
definition of $\chi_p$ yields that $\supp u\subseteq\Om_{p,\e}$,
and the linear size of  $\Om_{p,\e}$ is of order $\Odr(\e^2)$.

Let $x_p^{(0)}$ be a point in this support. By the smoothness,
the matrix $A_\e$ satisfies the identity
\begin{equation*}
A_\e(x)=A_\e(x_p^{(0)})+\e \widetilde{A}(x,p,\e),\qquad
|\widetilde{A}(x,p,\e)|\leqslant C,\quad x\in\Om_{p,\e},
\end{equation*}
where the constant $C$ is independent of  $\e$, $p$ and
$x\in\Om_{p,\e}$. Thus,
\begin{equation*}
-\sum\limits_{i,j=1}^{d}B_i^* A_\e(x_p^{(0)})B_j\p_{ij}
\mathbf{u}_p=\mathbf{f}_p+\e \sum\limits_{i,j=1}^{d} B_i^*
\widetilde{A}_\e B_j\p_{ij}\mathbf{u}_p,\quad x\in\Om_{p,\e}.
\end{equation*}
The left-hand side of this equation contains a differential
operator with constant coefficients that allows us to employ the
estimate (10.1) from \cite[Ch. I\!V, \S 10.1, Th. 10.1]{ADN1}
and to obtain
\begin{equation*}
\sum\limits_{i,j=1}^{d}\|\p_{ij}\mathbf{u}_p
\|_{L_2(\Om_{p,\e};\mathbb{C}^n)}\leqslant C\left(
\|\mathbf{f}_p\|_{L_2(\Om_{p,\e};\mathbb{C}^n)}+\e
\sum\limits_{i,j=1}^{d}\|\p_{ij}\mathbf{u}_p
\|_{L_2(\Om_{p,\e};\mathbb{C}^n)}\right),
\end{equation*}
where the constant $C$ is independent of $\e$ and $p$. In view
of (\ref{2.3a}) we conclude now that
\begin{align*}
&\sum\limits_{i,j=1}^{d}\|\p_{ij}\mathbf{u}_p
\|_{L_2(\Om_{p,\e};\mathbb{C}^n)}\leqslant C
\|\mathbf{f}_p\|_{L_2(\Om_{p,\e};\mathbb{C}^n)},
\\
&\sum\limits_{i,j=1}^{d}\|\p_{ij}\mathbf{u}
\|_{L_2(\mathbb{R}^d;\mathbb{C}^n)}^2=
\sum\limits_{i,j=1}^{d}\sum\limits_{p}\big\|\chi_p\p_{ij}\mathbf{u}
\big\|_{L_2(\Om_{p};\mathbb{C}^n)}^2 \leqslant C\sum\limits_{p}
\sum\limits_{i,j=1}^{d}\| \p_{ij}\mathbf{u}_p
\|_{L_2(\Om_{p};\mathbb{C}^n)}^2
\\
&\hphantom{\sum\limits_{i,j=1}^{d}\|\p_{ij}\mathbf{u}\|_{L_2(\mathbb{R}^d;}}
+
C\e^{-8}\sum\limits_{p}\|\mathbf{u}\|_{\H^1(\Om_{p};\mathbb{C}^n)}^2
\leqslant C\e^{-9} \left(\|\mathbf{f}
\|_{L_2(\mathbb{R}^d;\mathbb{C}^n)}^2+\|\mathbf{u}
\|_{\H^1(\mathbb{R}^d;\mathbb{C}^n)}^2\right).
\end{align*}
This inequality lead us to the statement of the lemma.
\end{proof}

\begin{lemma}\label{lm3.0}
Let $\mathbf{f}(x,\cdot)\in C_{per}^\b(\overline{\square})$ for
all $x\in\mathbb{R}^d$. The system
\begin{equation*}
B(\p_\xi)^*A(x,\xi)B(\p_\xi)
\mathbf{v}(x,\xi)=\mathbf{f}(x,\xi),\quad \xi\in\mathbb{R}^d,
\end{equation*}
has the $\square$-periodic in $\xi$ solution
$\mathbf{v}(x,\cdot)\in C_{per}^{2+\b}(\overline{\square})$
unique up to a constant (w.r.t. $\xi$) vector, if and only if
\begin{equation}\label{3.7}
\int\limits_\square \mathbf{f}(x,\xi)\di \xi=0,\quad
x\in\mathbb{R}^d.
\end{equation}
If this condition is satisfied, there exists the unique solution
$\mathbf{v}$ satisfying (\ref{3.7}), too. If $\mathbf{f}\in
W_p^k\big(\mathbb{R}^d;C_{per}^\b (\overline{\square})\big)$,
$p=2$, $p=\infty$, then $\mathbf{v}\in W_p^k\big(\mathbb{R}^d;
C_{per}^{2+\b}(\overline{\square})\big)$, and the estimate
\begin{equation*}
\|\mathbf{v}\|_{W_p^k(\mathbb{R}^d;C_{per}^{2+\b}(\overline{\square}))}\leqslant
C\|\mathbf{f}\|_{W_p^k(\mathbb{R}^d;C_{per}^\b(\overline{\square}))}
\end{equation*}
is valid.
\end{lemma}

This lemma is proved completely by analogy with Lemma~2.2 in
\cite{AA}.

\begin{lemma}\label{lm5.3}
For $\l$ close to $\l_0$, sufficiently small $\e$, and any
$\mathbf{f}\in L_2(\mathbb{R}^d;\mathbb{C}^n)$ the
representation
\begin{equation}\label{5.13}
(\mathcal{H}_\e-\l)^{-1}\mathbf{f}=\sum\limits_{i=1}^{N}
\frac{\boldsymbol{\Psi}_\e^{(i)}}{\l_\e^{(i)}-\l}
(\mathbf{f},\boldsymbol{\Psi}_\e^{(i)}
)_{L_2(\mathbb{R}^d;\mathbb{C}^n)}+\widetilde{\mathbf{u}}_\e,
\end{equation}
holds true, where $\boldsymbol{\Psi}_\e^{(i)}$ are the
eigenfunctions  associated  with $\l_\e^{(i)}$ and
orthonormalized in $L_2(\mathbb{R}^d;\mathbb{C}^n)$, while the
vector-function $\widetilde{\mathbf{u}}_\e$ satisfies the
uniform in $\e$ and $\l$ estimates
\begin{equation}\label{5.14}
\|\widetilde{\mathbf{u}}_\e\|_{L_2(\mathbb{R}^d;\mathbb{C}^n)}
\leqslant C\|\mathbf{f}\|_{L_2(\mathbb{R}^d;\mathbb{C}^n)},\quad
\|\widetilde{\mathbf{u}}_\e\|_{\H^2(\mathbb{R}^d;\mathbb{C}^n)}
\leqslant
C\e^{-5}\|\mathbf{f}\|_{L_2(\mathbb{R}^d;\mathbb{C}^n)}.
\end{equation}
\end{lemma}

\begin{proof}
The representation (\ref{5.13}) follows from \cite[Ch. V, \S
3.5, Formula (3.21)]{K}, where the vector-function
$\widetilde{\mathbf{u}}_\e$ is holomorphic w.r.t. $\l$ close to
$\l_0$ in the norm of $L_2(\mathbb{R}^d;\mathbb{C}^n)$. One can
make sure that
\begin{equation}\label{5.15}
\widetilde{\mathbf{u}}_\e=(\mathcal{H}_\e-\l)^{-1}
\widetilde{\mathbf{f}},\quad
\widetilde{\mathbf{f}}:=\mathbf{f}-\sum\limits_{i=1}^{N}
(\mathbf{f},\boldsymbol{\Psi}_\e^{(i)})_{L_2(\mathbb{R}^d;\mathbb{C}^n)}
\mathbf{\Psi}_\e^{(i)}.
\end{equation}
Let $\d$ be a sufficiently small fixed number such that
$\discspec(\mathcal{H}_0)\cap\{\l:
|\l-\l_0|\leqslant\d\}=\{\l_0\}$. By the convergences
$\l_\e^{(i)}\to\l_0$ for all $\e$ small enough we hence have
$\dist\big(\discspec(\mathcal{H}_\e),\{\l:
|\l-\l_0|=\d\}\big)\geqslant \d/2$. This inequality by \cite[Ch.
V, \S 3.5, Formula (3.16)]{K} and (\ref{5.15}) implies the
former of the estimates in (\ref{5.14}) for $|\l-\l_0|=\d$,
where the constant $C$ is independent of $\e$ and $\l$. By the
maximum modulus principle for the holomorphic functions  we
conclude that this estimate is valid for $|\l-\l_0|<\d$ as well.
The former estimate in (\ref{5.14}) follows now from
Lemma~\ref{lm2.2}.
\end{proof}

\section{Proof of Theorem~\ref{th1.3}}

First we  construct formally the asymptotic expansions for the
eigenvalues and the eigenfunctions. Then the justification of
these asymptotics will be adduced.

We employ the two-scale method \cite{BP} in formal constructing.
The asymptotics for the eigenvalues of $\mathcal{H}_\e$
converging to $\l_0$ are constructed as the series (\ref{1.16}),
while the asymptotics for the associated eigenfunctions are
sought as the series (\ref{1.18}). The aim of the formal
construction is to determine the coefficients of the series
(\ref{1.16}), (\ref{1.18}). The vector-functions
$\Psi_j^{(i)}=\Psi_j^{(i)}(x,\xi)$, $j\geqslant 1$, are sought
to be $\square$-periodic in $\xi$, and fast decaying as
$|x|\to+\infty$. We will specify their smoothness and behavior
at infinity in more details during the constructing.

We substitute the series (\ref{1.16}), (\ref{1.18}) into the
equation $\mathcal{H}_\e\boldsymbol{\psi}_\e^{(i)}=
\l_\e^{(i)}\boldsymbol{\psi}_\e^{(i)}$ and collect the
coefficients of the same powers of $\e$. As a result, we arrive
at the series of the equations
\begin{equation}\label{5.1}
B(\p_\xi)^*A B(\p_\xi)\boldsymbol{\Psi}_{j+2}^{(i)}=
-\mathcal{K}_{-1}\boldsymbol{\Psi}_{j+1}^{(i)}-
\mathcal{K}_0\boldsymbol{\Psi}_j^{(i)}+\l_0
\boldsymbol{\Psi}_j^{(i)} +
\sum\limits_{k=1}^{j}\l_k^{(i)}\boldsymbol{\Psi}_{j-k}^{(i)},
\quad (x,\xi)\in\mathbb{R}^{2d},
\end{equation}
where $j\geqslant -1$, $A=A(x,\xi)$, $V=V(x,\xi)$,
$\boldsymbol{\Psi}_j^{(i)}=\boldsymbol{\Psi}_j^{(i)}(x,\xi)$,
$q\geqslant 1$, $\boldsymbol{\Psi}_{j}^{(i)}\equiv0$, $j<0$. For
$j=-1$ the equation (\ref{5.1}) casts into
\begin{equation*}
B(\p_\xi)^*A B(\p_\xi)\boldsymbol{\Psi}_1^{(i)}=
-\mathcal{K}_{-1}\boldsymbol{\Psi}_0^{(i)},\quad
(x,\xi)\in\mathbb{R}^{2d}.
\end{equation*}
As it follows from (\ref{1.10}) and the definition of
$\mathcal{K}_{-1}$, a $\square$-periodic w.r.t. $\xi$ solution
of this equation is given by
\begin{equation}\label{5.2}
\boldsymbol{\Psi}_1^{(i)}(x,\xi)=\boldsymbol{\Upsilon}_1^{(i)}(x,\xi)+
\boldsymbol{\phi}_1^{(i)}(x),\quad
\boldsymbol{\Upsilon}_1^{(i)}:=(\L_1
B(\p)+\L_0)\boldsymbol{\Psi}_i^{(0)},
\end{equation}
where $\boldsymbol{\phi}_1^{(i)}$ is a vector-function to be
determined.

It follows from Lemma~\ref{lm3.0} that $\L_i\in
\mathcal{W}(\mathbb{R}^d;C_{per}^{2+\b}(\overline{\square}))$,
and thus the coefficients of $\mathcal{H}_0$ belong to
$\mathcal{W}(\mathbb{R}^d)$. Employing this fact and
differentiating the equation
$(\mathcal{H}_0-\l_0)\Psi_0^{(i)}=0$, one can easily make sure
that $\boldsymbol{\Psi}_0^{(i)}\in
\H^\infty(\mathbb{R}^d;\mathbb{C}^n)$, where
$\H^\infty(\mathbb{R}^d;\mathbb{C}^n):=\bigcap\limits_{k=1}^\infty
\H^k(\mathbb{R}^d;\mathbb{C}^n)$. It implies that
$\boldsymbol{\Upsilon}_1^{(i)}\in
\H^\infty(\mathbb{R}^d;C^{2+\b}_{per}(\overline{\square}))$.

We substitute (\ref{5.2}) into (\ref{5.1}) for $j=0$ to obtain
\begin{equation}\label{5.3}
B(\p_\xi)^* A B(\p_\xi)\boldsymbol{\Psi}_2^{(i)}
=-\mathcal{K}_{-1}\boldsymbol{\Upsilon}_1^{(i)}-
\mathcal{K}_0\boldsymbol{\Psi}_0^{(i)}
+\l_0\boldsymbol{\Psi}_0^{(i)} -
\mathcal{K}_{-1}\boldsymbol{\phi}_1^{(i)}, \quad
(x,\xi)\in\mathbb{R}^{2d}.
\end{equation}
In accordance with Lemma~\ref{lm3.0} the equation is uniquely
solvable in the class of $\square$-periodic w.r.t. $\xi$
vector-functions, if the solvability condition (\ref{3.7}) holds
true. In view of the identities
\begin{equation}
\begin{aligned}
&\int\limits_\square (B(\p_\xi)\L_0)^*A \di\xi=
\sum\limits_{i=1}^{d}\int\limits_\square a_i b_i
\frac{\p\L_1}{\p\xi_i}\di\xi,
\\
&\int\limits_\square  \big(B(\p_\xi)\L_0\big)^* A
B(\p_\xi)\L_0\di\xi=-\sum\limits_{i=1}^{d}\int\limits_\square
a_i b_i \frac{\p\L_0}{\p\xi_i}\di\xi,
\end{aligned}\label{3.9a}
\end{equation}
established in \cite{AA}, it is easy to check that this
solvability condition leads us to the equation
$(\mathcal{H}_0-\l_0\boldsymbol{)\Psi}_0^{(i)}=0$, which holds
true by the definition of $\Psi_0^{(i)}$. Hence, the
vector-function $\boldsymbol{\Psi}_2^{(i)}$ reads as follows
\begin{equation}\label{5.4}
\boldsymbol{\Psi}_2^{(i)}(x,\xi)=
\boldsymbol{\Upsilon}_2^{(i)}(x,\xi)+
(\L_1(x,\xi)B(\p)+\L_0(x,\xi))
\boldsymbol{\phi}_1^{(i)}(x)+\boldsymbol{\phi}_2^{(i)}(x),
\end{equation}
where  $\boldsymbol{\phi}_2^{(i)}$ is a vector-function to be
determined, and $\boldsymbol{\Upsilon}_2^{(i)}$  are
$\square$-periodic w.r.t. $\xi$ solutions to the equation
\begin{equation}\label{5.5}
B(\p_\xi)^* A B(\p_\xi)\boldsymbol{\Upsilon}_2^{(i)}=
-\mathcal{K}_{-1}\boldsymbol{\Upsilon}_1^{(i)}-
\mathcal{K}_0\boldsymbol{\Psi}_0^{(i)}
+\l_0\boldsymbol{\Psi}_0^{(i)}, \quad (x,\xi)\in\mathbb{R}^{2d},
\end{equation}
and satisfy (\ref{3.7}).  The equation is uniquely solvable,
since the right hand side of (\ref{5.3}) and the vector-function
$\mathcal{K}_{-1}\boldsymbol{\phi}_1^{(i)}$ satisfy (\ref{3.7}).
By Lemma~\ref{lm3.0},
$\boldsymbol{\Upsilon}_2^{(i)}\in\H^\infty(\mathbb{R}^d;
C^{2+\b}_{per}(\overline{\square}))$.

We substitute now (\ref{5.2}), (\ref{5.4}) into (\ref{5.1}) with
$j=1$,
\begin{equation}\label{5.8}
\begin{aligned}
B(\p_\xi)^* A B(\p_\xi)&\boldsymbol{\Psi}_3^{(i)}=
-\mathcal{K}_{-1}\boldsymbol{\Upsilon}_2^{(i)}
-\mathcal{K}_0\boldsymbol{\Upsilon}_1^{(i)}+
\l_0\boldsymbol{\Upsilon}_1^{(i)} - \mathcal{K}_{-1}(\L_1
B(\p_x)+\L_0)\boldsymbol{\phi}_1^{(i)}
\\
&- \mathcal{K}_0\boldsymbol{\phi}_1^{(i)}+
\l_0\boldsymbol{\phi}_1^{(i)}+
\l_1^{(i)}\boldsymbol{\Psi}_0^{(i)}-
\mathcal{K}_{-1}\boldsymbol{\phi}_2^{(i)}, \quad
(x,\xi)\in\mathbb{R}^{2d}.
\end{aligned}
\end{equation}
We write down the solvability condition (\ref{3.7}) for the
equation and take into account (\ref{3.9a}). It leads us to the
equation for $\boldsymbol{\phi}_1^{(i)}$,
\begin{equation}\label{5.6}
(\mathcal{H}_0-\l_0)\boldsymbol{\phi}_1^{(i)}
=\l_1^{(i)}\boldsymbol{\Psi}_0^{(i)} -\frac{1}{|\square|}
\int\limits_\square(\mathcal{K}_{-1}\boldsymbol{\Upsilon}_2^{(i)}+
\mathcal{K}_0\boldsymbol{\Upsilon}_1^{(i)})\di\xi.
\end{equation}
The right hand side of this equation is an element of
$\H^\infty(\mathbb{R}^d;\mathbb{C}^n)$. Since $\l_0$ is an
isolated eigenvalue of $\mathcal{H}_0$, the equation is solvable
in $\H^2(\mathbb{R}^d;\mathbb{C}^n)$, if and only if
\begin{equation}\label{5.7}
\left(\l_1^{(i)}\boldsymbol{\Psi}_0^{(i)} -\frac{1}{|\square|}
\int\limits_\square(\mathcal{K}_{-1}\boldsymbol{\Upsilon}_2^{(i)}+
\mathcal{K}_0\boldsymbol{\Upsilon}_1^{(i)})\di\xi,\boldsymbol{\Psi}_0^{(k)}
\right)_{L_2(\mathbb{R}^d;\mathbb{C}^n)}=0,\quad k=1,\ldots,N.
\end{equation}
If these conditions hold, the solution to (\ref{5.6}) is defined
uniquely up to a linear combination of
$\boldsymbol{\Psi}_0^{(i)}$.

\begin{lemma}\label{lm5.1}
The identities
\begin{equation*}
\frac{1}{|\square|}
(\mathcal{K}_{-1}\boldsymbol{\Upsilon}_2^{(i)}+
\mathcal{K}_0\boldsymbol{\Upsilon}_1^{(i)},\boldsymbol{\Psi}_0^{(k)}
)_{L_2(\mathbb{R}^d\times\square;\mathbb{C}^n)}=
\begin{cases}
0, & i\not=k,
\\
\tau_i,& i=k,
\end{cases}
\end{equation*}
hold true.
\end{lemma}
\begin{proof}
Integrating by parts and taking into account the equations
(\ref{1.10}), (\ref{5.5}) and the conditions (\ref{1.11}), we
obtain
\begin{align*}
&(\mathcal{K}_{-1}\boldsymbol{\Upsilon}_2^{(i)},
\boldsymbol{\Psi}_0^{(k)})_{L_2
(\mathbb{R}^d\times\square;\mathbb{C}^n)}
=(\boldsymbol{\Upsilon}_2^{(i)},\mathcal{K}_{-1}
\boldsymbol{\Psi}_0^{(k)})_{L_2
(\mathbb{R}^d\times\square;\mathbb{C}^n)}
\\
&=-(\boldsymbol{\Upsilon}_2^{(i)}, B(\p_\xi)^*A
B(\p_\xi)\boldsymbol{\Upsilon}_1^{(k)})_{L_2
(\mathbb{R}^d\times\square;\mathbb{C}^n)}=-(B(\p_\xi)^*A
B(\p_\xi)\boldsymbol{\Upsilon}_2^{(i)},
\boldsymbol{\Upsilon}_1^{(k)})_{L_2
(\mathbb{R}^d\times\square;\mathbb{C}^n)}
\\
&=(\mathcal{K}_{-1}\boldsymbol{\Upsilon}_1^{(i)},
\boldsymbol{\Upsilon}_1^{(k)})_{L_2
(\mathbb{R}^d\times\square;\mathbb{C}^n)}+
(\boldsymbol{\Psi}_0^{(i)}, \mathcal{K}_{0}
\boldsymbol{\Upsilon}_1^{(k)})_{L_2
(\mathbb{R}^d\times\square;\mathbb{C}^n)}.
\end{align*}
By the definition of $T$ and $\boldsymbol{\Psi}_0^{(i)}$ it
proves the lemma.
\end{proof}

In view of the definition of $\boldsymbol{\Psi}_i^{(0)}$ the
identities (\ref{5.7}) hold true, if the numbers $\l_1^{(i)}$
are chosen in accordance with (\ref{1.17}). The corresponding
solution of the equation (\ref{5.6}) reads as follows
\begin{equation}\label{5.9}
\boldsymbol{\phi}_1^{(i)}=\widetilde{\boldsymbol{\phi}}_1^{(i)}
+\sum\limits_{p=1}^{N}
S_{ik}^{(1)}\boldsymbol{\Psi}_0^{(k)},\quad
\big(\widetilde{\boldsymbol{\phi}}_1^{(i)},
\boldsymbol{\Psi}_0^{(k)}\big)_{L_2(\mathbb{R}^d;\mathbb{C}^n)}=0,
\quad k=1,\ldots,N
\end{equation}
where $S_{ip}^{(1)}$ are numbers, and
$\widetilde{\boldsymbol{\phi}}_1^{(i)}\in
\H^\infty(\mathbb{R}^d;\mathbb{C}^n)$. The latter belonging can
be justified easily by differentiating (\ref{5.6}). The solution
of (\ref{5.8}) is given by
\begin{equation}\label{5.10}
\boldsymbol{\Psi}_3^{(i)}(x,\xi)=\boldsymbol{\Phi}_3^{(i)}(x,\xi)
+(\L_1(x,\xi)B(\p_x)+\L_0(x,\xi))\boldsymbol{\phi}_2^{(i)}(x)+
\sum\limits_{k=1}^{N}
S_{ik}^{(1)}\boldsymbol{\Upsilon}_2^{(k)}(x,\xi)+
\boldsymbol{\phi}_3^{(i)}(x),
\end{equation}
where $\boldsymbol{\phi}_3^{(i)}$ is the vector-function, and
$\boldsymbol{\Phi}_3^{(i)}\in
\H^\infty(\mathbb{R}^d;C^{2+\b}_{per}(\overline{\square}))$ is a
solution to
\begin{align*}
B(\p_\xi)^* A B(\p_\xi)\boldsymbol{\Phi}_3^{(i)}=&
-\mathcal{K}_{-1}\boldsymbol{\Upsilon}_2^{(i)} -
\mathcal{K}_0\boldsymbol{\Upsilon}_1^{(i)}+
\l_0\boldsymbol{\Upsilon}_1^{(i)}- \mathcal{K}_{-1}(\L_1
B(\p_x)+\L_0)\widetilde{\boldsymbol{\phi}}_1^{(i)}
\\
&-\mathcal{K}_0\widetilde{\boldsymbol{\phi}}_1^{(i)}+
\l_0\widetilde{\boldsymbol{\phi}}_1^{(i)}+
\l_1^{(i)}\boldsymbol{\Psi}_0^{(i)}, \quad
(x,\xi)\in\mathbb{R}^{2d}.
\end{align*}

Let us show how to determine $S_{ip}^{(1)}$ and $\l_2^{(i)}$. We
substitute (\ref{5.2}), (\ref{5.4}), (\ref{5.9}), (\ref{5.10})
into (\ref{5.1}) with $j=2$. The solvability condition
(\ref{3.7}) for this equation yields
\begin{align}
&
\begin{aligned}
(\mathcal{H}_0-\l_0)\boldsymbol{\phi}_2^{(i)}=&\l_1^{(i)}
\sum\limits_{p=1}^{N}S_{ip}^{(1)}\boldsymbol{\Psi}_0^{(p)}
-\frac{1}{|\square|}\sum\limits_{k=1}^{N} S_{ik}^{(1)}
\int\limits_\square(\mathcal{K}_{-1}\boldsymbol{\Upsilon}_2^{(k)}+
\mathcal{K}_0\boldsymbol{\Upsilon}_1^{(k)})\di\xi
\\
&+\l_1^{(i)}\widetilde{\boldsymbol{\phi}}_1^{(i)}
+\l_2^{(i)}\boldsymbol{\Psi}_0^{(i)}-\mathbf{g}_2^{(i)},\quad
x\in\mathbb{R}^d,
\end{aligned}\label{5.11}
\\
& \mathbf{g}_2^{(i)}:= \frac{1}{|\square|} \int\limits_\square
\big(\mathcal{K}_{-1}\boldsymbol{\Phi}_3^{(i)}+\mathcal{K}_0
\boldsymbol{\Upsilon}_2^{(i)}+ \mathcal{K}_0(\L_1
B(\p_x)+\L_0)\widetilde{\boldsymbol{\phi}}_1^{(i)}\big)\di\xi.
\nonumber
\end{align}
The right hand side of the obtained equation belongs to
$\H^\infty(\mathbb{R}^d;\mathbb{C}^n)$. Now we write the
solvability condition for (\ref{5.11}) and take into account
Lemma~\ref{lm5.1} and (\ref{1.17}). It leads us to
\begin{equation*}
(\tau_1^{(i)}-\tau_1^{(k)})S_{ik}^{(1)}+ \l_2^{(i)}\d_{ik}
-(\mathbf{g}_2^{(i)},\boldsymbol{\Psi}_0^{(k)})_{L_2(\mathbb{R}^d;\mathbb{C}^n)}=0,\quad
k=1,\ldots,N,
\end{equation*}
where $\d_{ik}$ is the Kronecker delta. By the assumption
$\tau_i\not=\tau_j$, $i\not=j$, it implies that
\begin{equation}\label{5.12}
S_{ik}^{(1)}=\frac{(\mathbf{g}_2^{(i)},
\boldsymbol{\Psi}_0^{(k)})_{L_2(\mathbb{R}^d;\mathbb{C}^n)}}
{\tau_i-\tau_k},\quad k\not=i,\qquad
\l_2^{(i)}=(\mathbf{g}_2^{(i)},
\boldsymbol{\Psi}_0^{(i)})_{L_2(\mathbb{R}^d;\mathbb{C}^n)}.
\end{equation}
Without loss of generality we let $S_{ii}^{(1)}=0$. The formulas
(\ref{1.19}) are proven.

The construction of the other terms of the series (\ref{1.16}),
(\ref{1.18}) is carried out by the same scheme. The results are
in the following lemma which can be easily proved  by induction.
\begin{lemma}\label{lm5.2}
There exist unique solutions to (\ref{5.1}) and the uniquely
determined numbers $\l_{j}^{(i)}$, which read as follows
\begin{align*}
\boldsymbol{\Psi}_j^{(i)}(x,\xi)=&\boldsymbol{\Phi}_j^{(i)}(x,\xi)
+(\L_1(x,\xi)B(\p_x)+
\L_0(x,\xi))\boldsymbol{\phi}_{j-1}^{(i)}(x)
\\
&+ \sum\limits_{k=1}^{N}
S_{ik}^{(j-2)}\boldsymbol{\Upsilon}_k^{(2)}(x,\xi)+
\boldsymbol{\phi}_j^{(i)}(x),
\\
\boldsymbol{\phi}_j^{(i)}(x)=&
\widetilde{\boldsymbol{\phi}}_j^{(i)}(x)+ \sum\limits_{k=1}^{N}
S_{ik}^{(j)}\boldsymbol{\Psi}_0^{(k)}(x),
\end{align*}
where $\boldsymbol{\Phi}_j^{(i)}\in\H^\infty(\mathbb{R}^d;
C^{2+\b}_{per}(\overline{\square}))$ are the solutions to
\begin{align*}
B(\p_\xi)^*A
B(\p_\xi)\boldsymbol{\Phi}_j^{(i)}=&-\mathcal{K}_{-1}
\left(\boldsymbol{\Phi}_{j-1}^{(i)}+(\L_1 B(\p_x)+\L_0)
\widetilde{\boldsymbol{\phi}}_{j-2}^{(i)}+\sum\limits_{k=1}^{N}
S_{ik}^{(j-3)}\boldsymbol{\Upsilon}_2^{(k)}\right)
\\
&-\mathcal{K}_0 \left(\boldsymbol{\Phi}_{j-2}^{(i)}+(\L_1
B(\p_x)+\L_0)
\widetilde{\boldsymbol{\phi}}_{j-3}^{(i)}+\sum\limits_{k=1}^{N}
S_{ik}^{(j-4)}\boldsymbol{\Upsilon}_2^{(k)}+
\widetilde{\boldsymbol{\phi}}_{j-2}^{(i)}\right)
\\
&+\sum\limits_{k=0}^{j-2}\l_k^{(i)}\boldsymbol{\Psi}_{j-k-2},\quad
(x,\xi)\in\mathbb{R}^{2d},
\end{align*}
and satisfy (\ref{3.7}), $\l_0^{(i)}:=\l_0$. The
vector-functions $\widetilde{\boldsymbol{\phi}}_j^{(i)}
\in\H^\infty(\mathbb{R}^d;\mathbb{C}^n)$ are the solutions to
the equations
\begin{align*}
&(\mathcal{H}_0-\l_0)\widetilde{\boldsymbol{\phi}}_j^{(i)}
=\l_1^{(i)}
\sum\limits_{k=1}^{N}S_{ik}^{(j)}\boldsymbol{\Psi}_0^{(k)} -
\frac{1}{|\square|}\sum\limits_{k=1}^{N} S_{ik}^{(j)}
\int\limits_\square(\mathcal{K}_{-1}\boldsymbol{\Upsilon}_2^{(i)}+
\mathcal{K}_0\boldsymbol{\Upsilon}_1^{(i)})\di\xi
\\
&\hphantom{(\mathcal{H}_0-\l_0)\widetilde{\boldsymbol{\phi}}_j^{(i)}
=} +\l_j^{(i)}\boldsymbol{\Psi}_0^{(i)}+\sum\limits_{k=2}^{j-1}
\l_k^{(i)}\boldsymbol{\phi}_{j-k}^{(1)}- \mathbf{g}_j^{(i)},
\\
&\mathbf{g}_{j}^{(i)}:=\frac{1}{|\square|}\int\limits_\square
\left(\mathcal{K}_{-1}\Phi_{j+1}^{(i)}+\mathcal{K}_0 \big(
\Phi_{j+2}^{(i)}+(\L_1 B(\p_x)+\L_0)
\widetilde{\boldsymbol{\phi}}_{j-1}^{(i)}\big)+
\sum\limits_{k=1}^{N}S_{ik}^{(j-2)}\boldsymbol{\Upsilon}_2^{(k)}\right)
\di\xi,
\end{align*}
being orthogonal to $\boldsymbol{\Psi}_0^{(k)}$, $k=1,\ldots,N$,
in $L_2(\mathbb{R}^d;\mathbb{C}^n)$. The numbers $S_{ik}^{(j)}$
and $\l_j^{(i)}$ are determined by
\begin{equation*}
S_{ik}^{(j)}=\frac{(\mathbf{g}_j^{(i)},
\boldsymbol{\Psi}_0^{(k)})_{L_2(\mathbb{R}^d;\mathbb{C}^n)}}
{\tau_i-\tau_k},\quad i\not=k, \qquad S_{ii}^{(j)}=0,
\qquad\l_j^{(i)}=(\mathbf{g}_j^{(i)},
\boldsymbol{\Psi}_0^{(i)})_{L_2(\mathbb{R}^d;\mathbb{C}^n)}.
\end{equation*}
\end{lemma}

Thus, the coefficients of the series (\ref{1.16}), (\ref{1.18})
are determined that completes the formal constructing of the
asymptotics.

We denote
\begin{equation*}
\l_{\e,k}^{(i)}:=\l_0+\sum\limits_{j=1}^{k}\e^j\l_j^{(i)},\quad
\boldsymbol{\psi}_{\e,k}^{(i)}(x):=\boldsymbol{\Psi}_{0}^{(i)}(x)
+\sum\limits_{j=1}^{k}\e^j\boldsymbol{\Psi}_j^{(i)}
\left(x,\frac{x}{\e}\right).
\end{equation*}

The next lemma follows directly from Lemma~\ref{lm5.2} and the
equations (\ref{5.1}).

\begin{lemma}\label{lm5.4}
The function $\boldsymbol{\psi}_{\e,k}^{(i)}\in
\H^\infty(\mathbb{R}^d;\mathbb{C}^n)$ solves the equation
$(\mathcal{H}_\e-\l_{\e,k}^{(i)})\boldsymbol{\psi}_{\e,k}^{(i)}
=\mathbf{f}_{\e,k}^{(i)}$, where the function
$\mathbf{f}_{\e,k}^{(i)}\in L_2(\mathbb{R}^d;\mathbb{C}^n)$
satisfies an uniform in $\e$ estimate
\begin{equation*}
\|\mathbf{f}_{\e,k}^{(i)}\|_{L_2(\mathbb{R}^d;\mathbb{C}^n)}\leqslant
C\e^{k-1}.
\end{equation*}
\end{lemma}

Basing on Lemmas~\ref{lm5.3},~\ref{lm5.4} and proceeding
completely in the similar way as in the proof of Lemma~4.3 and
Theorem~1.1 in \cite{B2}, one can show now that the eigenvalues
$\l_\e^{(i)}$ and the associated eigenfunctions satisfy the
asymptotic expansions (\ref{1.16}), (\ref{1.18}).

\end{document}